\documentclass[12pt,a4paper]{amsart}
 \usepackage[titletoc,toc,title]{appendix}
\usepackage{t1enc}
\usepackage[latin1]{inputenc}
\usepackage{amssymb}
\usepackage{amsthm}
\usepackage{amsmath}
\usepackage{latexsym}
\usepackage{array}
\usepackage{mathrsfs}
\usepackage[all]{xy}
\usepackage{graphics}
\usepackage{xcolor}
\usepackage{bbm}
\usepackage{stmaryrd}
\usepackage{color}
\usepackage{enumerate}
\usepackage{tikz}
\usetikzlibrary{matrix}

\title{Non-compact form of the Elementary Discrete Invariant}
\author{Rapha\"{e}l Fino}
\address
{Instituto de Matem\'{a}ticas \\
Ciudad Universitaria\\
UNAM\\
DF 04510, M\'{e}xico}
\address
{{\it Web page:}
{\tt http://www.matem.unam.mx/fino}}
\email {fino {\it at} im.unam.mx}
\numberwithin{equation}{section}
\theoremstyle{definition}
\newtheorem{defi}[equation]{Definition}

\newtheorem{rem}[equation]{Remark}

\newtheorem{lemme}[equation]{Lemma}

\newtheorem{prop}[equation]{Proposition}
\newtheorem{thm}[equation]{Theorem}
\newtheorem{cor}[equation]{Corollary}

\begin{document}

\begin{abstract}
We determine the non-compact form of Vishik's Elementary Discrete Invariant for quadrics.

\end{abstract}

\maketitle

\tableofcontents

\section{Introduction}

Let $X$ be a smooth projective quadric 
of dimension $n$ 
over a field $F$ associated with a non-degenerate $F$-quadratic form $q$.
% and let $K/F$ be a splitting field extension of $q$.
The \emph{splitting pattern} of $X$ is a discrete
 invariant which measures what are the possible Witt indices
of $q_E$ over all field extensions $E/F$ (see \cite{HR} and \cite{Kne}).
On the other hand, the \emph{motivic decomposition type} of $X$ is a discrete
 invariant which measures in what pieces the Chow motive of $X$
can be decomposed. 
Moreover, Alexander Vishik noticed in \cite{Vi04} that the study of the interaction between these two invariants
provides further information about both of them.

For this reason, he introduced the \emph{Generic Discrete Invariant} of quadrics, a bigger discrete invariant containing the splitting pattern and the motivic decomposition type invariants as faces, see \cite{Grass} and \cite{ViICTP}.
The Generic Discrete Invariant $GDI(X)$ is defined as follows.
Let $K/F$ be a splitting field extension of $q$. Let us denote $[n/2]$ as $d$. 
For any $i\in \{0,\dots, d\}$, we write $G_i$ for the grassmannian 
of $i$-dimensional totally $q$-isotropic subspaces (in particular $G_0$ is the quadric $X$).
Then $GDI(X)$ is the collection of the subalgebras of rational elements
\[\overline{\text{Ch}}^{\ast}(G_i):=\text{Image}\left(\text{Ch}^{\ast}({G_i})\rightarrow \text{Ch}^{\ast}({G_i}_{K})\right)\]
for $i\in \{0,\dots, d\}$, where $\text{Ch}$ stands for the Chow ring
with $\mathbb{Z}/2\mathbb{Z}$-coefficients (an algebraic cycle already defined at the level of the base field $F$ is called \emph{rational}).

In his paper \cite{uINV} dedicated to the Kaplansky's conjecture on
the $u$-invariant of a field, A.\,Vishik 
used the \textit{Elementary Discrete Invariant} of quadrics, a handier invariant than the $GDI$ as it only deals with
some particular cycles in $\text{Ch}^{\ast}({G_i}_{K})$.
More precisely, 
%for any $I\subset \{0,\dots, d\}$, we write $\mathcal{F}(I)$ for the partial orthogonal
%flag variety associated with $q$. So, 
for any $i\in \{0,\dots, d\}$, 
%we write $G_i$ for the grassmannian 
%of $i$-dimensional totally isotropic subspaces (in particular $G_0$ is the quadric $X$) and 
we denote by $\mathcal{F}(0,i)$ the partial orthogonal flag variety of $q$-isotropic lines contained in 
$i$-dimensional totally $q$-isotropic subspaces. 
One can consider the diagram
\[\xymatrix{
X & \mathcal{F}(0,i) \ar[l]^{\pi_{(\underline{0},i)}} \ar[r]_{\pi_{(0,\underline{i})}} & G_i,
}\]
given by the natural projections
and, for $0\leq j \leq d$, we set
\[Z^i_{n-i-j}:={\pi_{(0,\underline{i})}}_{\ast} \circ \pi_{(\underline{0},i)}^{\ast}(l_{j})\in \text{CH}^{n-i-j}({G_i}_{K}),\]
where $\text{CH}$ stands for the Chow ring
with $\mathbb{Z}$-coefficients and $l_j$ is the class in $\text{CH}_j(X_K)$ of a $j$-dimensional totally isotropic subspace
of $\mathbb{P}\left((V_q)_K\right)$ (with $V_q$ the $F$-vector space associated with $q$).
We set $z^i_{n-i-j}:=Z^i_{n-i-j}\;(\text{mod}\;2) \in \text{Ch}^{n-i-j}({G_i}_{K})$, with $\text{Ch}$ the Chow ring
with $\mathbb{Z}/2\mathbb{Z}$-coefficients. 
%(denoted respectively as $Z^{\text{\boxed{i-d}}}_j$ and $z^{\text{\boxed{i-d}}}_j$ in \cite{uINV}). 
The cycles $z^i_{n-i-j}$ are the elementary classes defining
%corresponding to the main antidiagonal when representing 
the Elementary Discrete Invariant $EDI(X)$:

\begin{defi}
The \emph{Elementary Discrete Invariant} $EDI(X)$ is the collection of subsets $EDI(X,i)$
consisting of those integers $m$ such that $z^i_m$ is rational.
\end{defi}

Furthermore, for any $r\geq 1$, the Chow motive of $X^r$ with $\mathbb{Z}/2\mathbb{Z}$-coefficients
decomposes into a direct sum of shifts of the motive of some $G_i$, see \cite[Corollary 91.8]{EKM}.
Therefore, to know $GDI(X)$
is the same as to know
\[\overline{\text{Ch}}^{\ast}(X^r):=\text{Image}\left(\text{Ch}^{\ast}(X^r)\rightarrow \text{Ch}^{\ast}(X^r_K)\right)\]
for all $r\geq 1$. Hence, the collection of the latter
subalgebras constitutes a \emph{non-compact}
(in the sense that one has to consider infinitely many objects)
form of $GDI(X)$. For the same reason, there exists a non-compact form of $EDI(X)$
(with defining cycles living in $\text{Ch}^{\ast}(X^r_K)$), which we determine in the current note:
for any $i\in \{0,\dots, d\}$, let us denote by 
$\text{sym}: \text{CH}^{\ast}(X^{i+1})\rightarrow \text{CH}^{\ast}(X^{i+1})$
the homomorphism $\Sigma_{s \in S_{i+1}}s_{\ast}$, where
$s:X^{i+1}\rightarrow X^{i+1}$ is the isomorphism associated with
a permutation $s$. For $0\leq j \leq d$, we set 
\[\rho_{i,j}:=\text{sym}\left((\times_{k=0}^{i-1}h^k)\times l_j\right)\in \text{CH}^{n-j+i(i-1)/2}(X_{K}^{i+1}),\]
where $\times$ is the external product
and $h^k$ is the $k$-th power of the hyperplane section
class $h\in \text{CH}^1(X)$ (always rational).
%Note that $\rho_1$ is
%the \textit{ Rost correspondence} $1\times l_0+l_0\times 1$ of $X$ in the sense of \cite[\S 80]{EKM}
%(we refer to \cite[\S 62]{EKM} for an introduction to correspondences). Note that the rationality of %$\rho_i$ implies the rationality
%of $\rho_j$ for any $j>i$. Note also that $\rho_0=l_0=Z^0_n$.
Note that $\rho_{0,j}=Z^0_{n-j}=l_j$.
The symmetric cycles $\rho_{i,j}\;(\text{mod}\;2)$ are the classes defining the non-compact form of $EDI(X)$:

\begin{thm}
 \textit{Let $1\leq i \leq d $ and $0\leq j \leq d$. The cycle $z^{i}_{n-i-j}$ is rational if and only if the cycle} $\rho_{i,j}\;(\text{mod}\;2)$ \textit{is rational.}
\end{thm}

Theorem 1.2 reduces certain questions about rationality of algebraic cycles on orthogonal grassmannians to the sole level of quadrics.
For example, it allows one to reformulate both Vishik's conjecture \cite[Conjecture 3.11]{GPQ}
and the conjecture \cite[Conjecture 0.13]{ViICTP} on the \emph{dimensions of Bruno Kahn}.
%and these reformulations may help prove these conjectures.

\medskip

In Section 2, we introduce some basic tools which are required in Section 3, where we prove
Theorem 1.2, using mainly compositions of correspondences and 
Chern classes of vector bundles over orthogonal grassmannians.

\section{Preliminaries}

In this section, we continue to use notation introduced in Section\,1.

\subsection{Rational cycles on powers of quadrics}

We refer to \cite[\S 68]{EKM} for an introduction to cycles on powers of quadrics.
For any $1\leq i \leq d$ and $0\leq j \leq i-1$, we set
\[\Delta_{i,j}:=\text{sym}\left((\times_{k=0}^{i-1}h^k)\times l_j\right)
+\sum_{m=i}^d\text{sym}\left((\times_{\substack{ k=0\\ k\neq j  }}^{i-1}h^k)\times h^m \times l_m\right)\]
in $\text{Ch}^{n-j+i(i-1)/2}(X_{K}^{i+1})$.
If $n=2d$, we choose an orientation $l_d$ of the quadric.

\begin{lemme}
\textit{For any $1\leq i \leq d$ and $0\leq j \leq i-1$, the cycle} $\Delta_{i,j}$
\textit{is rational.}
\end{lemme}

\begin{proof}[Proof] We proceed by induction on $i$. 
In $\text{Ch}^{n}(X_{K}^{2})$, the cycle $\Delta_{1,0}$
or $\Delta_{1,0}+h^d\times h^d$, depending on whether $l_d^2=0$ or not, is the class of the diagonal. 
Therefore , the cycle $\Delta_{1,0}$ is rational.
Let $\sigma \in S_{i+1}$ be a cyclic permutation (with $i\geq 2$). For $0\leq j \leq i-2$, the induction hypothesis step is provided by the identity
\[\Delta_{i,j}=\sum_{l=0}^{i}\sigma_{\ast}^{l}(\Delta_{i-1,j}\times h^{i-1})\;\;\text{in}\;\; \text{Ch}(X_{K}^{i+1}).\]
It remains to show that the cycle $\Delta_{i,i-1}$ is rational to complete the proof.
In $\text{Ch}(X_{K}^{i+1})$, one has 
\[\begin{array}{lll}\Delta_{i,i-1}& = &\sum_{m=i-1}^d\text{sym}\left((\times_{k=0}^{i-2}h^k)\times l_m\times h^m\right)  \\
& & \\
& = & \sum_{m=0}^d\text{sym}\left((\times_{k=0}^{i-2}h^k)\times l_m\times h^m\right)\end{array}\]
and the latter sum can be rewritten as
\[\sum_{s\in A_{i+1}} s_{\ast}\left((\times_{k=0}^{i-2}h^k)\times \Delta_{1,0} \right).\]
Thus, the cycle $\Delta_{i,i-1}$ is rational.

\end{proof}

\subsection{Correspondences}

We refer to \cite[\S 62]{EKM} for an introduction to Chow-correspondences.

For any $1\leq i \leq d$, we denote by $\theta_i$ the class of the subvariety 
\[\{(y,x_1,\dots,x_{i+1})\;|\;x_1,\dots x_{i+1} \in y\}\subset G_i \times X^{i+1}\]
in $\text{CH}(G_i\times X^{i+1})$ and we view the cycle $\theta_i$ as a correspondence
$G_i\rightsquigarrow  X^{i+1}$.

We set
\begin{equation}
\eta_i:= \prod_{k=1}^i  \left(\text{Id}_{G_i}\times p_{X^i_k} \right)^{\ast}\left([\mathcal{F}(i,0)]\right) \in \text{CH}(G_i\times X^i),
\end{equation}
with $p_{X^i_k}$ the projection from $X^i$ to the $k$-th coordinate.
For any integer $i\leq s \leq d$, we write
\[W_{s-i}^i:={\pi_{(0,\underline{i})}}_{\ast} \circ\pi_{(\underline{0},i)}^{\ast}(h^{s})\in \text{CH}^{s-i}(G_i),\]
and $w^i_{s-i}:=W^i_{s-i}\;(\text{mod}\;2)\in \text{Ch}^{s-i}(G_i)$.
Since the variety $X_K$ is cellular, the cycle $[\mathcal{F}(i,0)]$ decomposes as 
\begin{equation} [\mathcal{F}(i,0)]=\sum_{s=0}^d z^i_{n-i-s}\times h^s +
\sum_{s=i}^d w^i_{s-i}\times l_s  \;\text{in}\; \text{Ch}({G_i}_K\times X_K),
\end{equation}
where $l_d$ has to be replaced by the other class $l_d'$ of maximal totally isotropic subspaces if
$n=2d$ and $l_d^2$ is not zero, i.e., if $4$ divides $n$ (see \cite[Theorem 66.2]{EKM}).
\medskip

The two following lemmas, where we write $p$ with underlined target for projections, 
can be proven the same way
 \cite[Lemmas 3.2 and 3.10]{sym} have been proven but with 
$Z^i_{n-i-j}$ (resp. $z^i_{n-i-j}$) instead of $Z^i_{n-i}$ (resp. $z^i_{n-i}$).

\begin{lemme}
\textit{For any $1\leq i \leq d$, $0\leq j \leq d$ and}
$x\in \text{CH}(X_{K})$\textit{, one has}
\begin{multline*}\left((\theta_i)_{\ast}(Z^i_{n-i-j})\right)_{\ast}(x)= \\
{p_{G_i\times \underline{X^{i}}}}_{\;\ast}\left( p_{\underline{G_i}\times X^{i}}^{\ast}\left(
{\pi_{(0,\underline{i})}}_{\ast} \circ \pi_{(\underline{0},i)}^{\ast}(x)   \cdot Z^i_{n-i-j}\right) \cdot \eta_i  \right),\end{multline*}
\textit{where the cycle $(\theta_i)_{\ast}(Z^i_{n-i-j})$ is viewed as a correspondence 
$X_{K}\rightsquigarrow X^{i}_{K}$.}
\end{lemme}

For any $1\leq i \leq d$, 
we write $\mathcal{F}(i-1,i)$ for the partial orthogonal flag variety of $(i-1)$-dimensional totally isotropic subspaces contained in $i$-dimensional totally isotropic subspaces
and we consider the diagram
\[\xymatrix{
G_{i-1} & \mathcal{F}(i-1,i) \ar[l]^{\pi_{(\underline{i-1},i)}} \ar[r]_{\,\,\,\,\,\,\,\,\,\,\pi_{(i-1,\underline{i})}} & G_i,
}\]
given by the natural projections.

\begin{lemme}
\textit{For any $2\leq i\leq d$, $0\leq j \leq d$ and $i\leq m \leq d$, the cycle} 
\[{p_{G_i\times \underline{X^{i}}}}_{\;\ast}\left(w^i_{m-i}   \cdot z^i_{n-i-j} \cdot \eta_i \right)\in \text{Ch}(X^i_K),\] 
\textit{where we abuse notation and write $\eta_i$ for} $\eta_i \;(\text{mod}\;2)$\textit{, can be rewritten as}
\[
\sum_{s=0}^m \sum_{k=\text{max}(i-s,0)}^{\text{min}(m-s,i)}
 {p_{G_{i-1}\times \underline{X^{i-1}}}}_{\;\ast}\left( 
w^{i-1}_{m-s-k}\cdot \sigma_{i-1}^k   \cdot z^{i-1}_{n-i+1-j} \cdot \eta_{i-1}  \right)\times h^s ,\]
\textit{with} $\sigma_{i-1}^k={\pi_{(\underline{i-1},i)}}_{\ast} \circ \pi_{(i-1,\underline{i})}^{\ast}
(z^{i}_{n-2i+k}) \in \text{Ch}^j({G_{i-1}}_K)$.
\end{lemme}

%The following observation is quite usefull in our situation (see Corollary 2.7).

%\begin{prop}
%\textit{For any $1\leq i \leq d$ and $0\leq k \leq i$, one has}
%\[(z^i_{n-i-k})^2=0.\]
%\end{prop}

%\begin{proof}

%\end{proof}

%\begin{cor}
%\textit{For any $1\leq i \leq d$ and $ j>i$, the cycle}
%\[{p_{G_i\times \underline{X^{i}}}}_{\;\ast}\left( z^i_{n-i-j} \cdot \eta_i \right)\in \text{Ch}(X^i_K)\]
%\textit{is trivial}. 
%\end{cor}

%\begin{proof}

%\end{proof}

\section{Equivalence}
In this section, we continue to use notation and material introduced in the previous sections and
we prove Theorem 1.2.

\medskip

For $1\leq i \leq d$ and $0\leq j \leq i-1$,
we set 
\[\alpha_{i,j}:= (\theta_i)_{\ast}(Z^i_{n-i-j})+\rho_{i,j}   \in\text{CH}(X_{K}^{i+1}),\]
and we view the cycle  $\alpha_{i,j}$ as a correspondence $X_{K}\rightsquigarrow X^{i}_{K}$.

\begin{prop}
\textit{One has}
\[\left( \alpha_{i,j} \;(\text{mod}\;2)\right)_{\ast}(h^m)=\left\{\begin{array}{ll} 
\text{sym}\left(\times_{k=0}^{i-1}h^k\right) & \textit{if} \:\: m=j  \textit{;}\\
\text{sym}\left((\times_{\substack{ k=0\\ k\neq j  }}^{i-1} h^k)\times h^m\right) & \textit{if}\:\: i \leq m \leq d \textit{;} \\
0 & \textit{otherwise.} 
\end{array} \right. \]
\end{prop}

\begin{proof}[Proof]
For any $x\in \text{CH}^m(X_{K})$ with $m\leq i-1$, the cycle 
${\pi_{(0,\underline{i})}}_{\ast} \circ \pi_{(\underline{0},i)}^{\ast}(x)$
is trivial by dimensional
reasons. Thus, by Lemma 2.4, the cycle $\left((\theta_i)_{\ast}(Z^i_{n-i-j})\right)_{\ast}(x)$ is also trivial.
Therefore, since $(\rho_{i,j})_{\ast}(h^m)=\text{sym}\left(\times_{k=0}^{i-1}h^k\right) $ if $m=j$ and is trivial otherwise,
one get the conclusion of Proposition 3.1 for the cases
 $m\leq i-1$.

Moreover, for $i\leq m\leq d$, Lemma 2.4 provides the identity
\begin{equation}\left( \alpha_{i,j} \;(\text{mod}\;2)\right)_{\ast}(h^m)=
{p_{G_i\times \underline{X^{i}}}}_{\;\ast}\left( 
w^i_{m-i}   \cdot z^i_{n-i-j} \cdot \eta_i \right) \; \text{in}\; \text{Ch}(X^i_K).
\end{equation}
We prove the cases $i\leq m\leq d$ of Proposition 3.1 by descending induction on $i$.
The base of the descending induction $i=d$ (so $i=m=d$) is obtained by combining 
the identities (2.2), (2.3)  and (3.2) for $i=d$ (recall also that, by \cite[Proposition 2.1]{uINV}, one has $W^i_0=1$ for any $0\leq i \leq d$) 
with the fact that, for any integers $0\leq a_0\leq a_1 \leq \cdots \leq a_{e} \leq d$, with $e\leq d$, one has
\[\text{deg}\left(\prod_{k=0}^e z^d_{n-d-a_k}    \right)=\left \{\begin{array}{ll} 1 & \;\;\text{if}\;\;              \{a_0, a_1, \dots , a_e\}=\{0,1,\dots , d\}\; ; \\
 0 &\;\; \text{otherwise} ,\end{array}\right.\]
where $\text{deg}:\text{Ch}({G_d}_K)\rightarrow \text{Ch}(\text{Spec}(K))=\mathbb{Z}/2\mathbb{Z}$ is 
the homomorphism associated with the push-forward of the structure morphism, see \cite[Lemma 87.6]{EKM}.

Let $2\leq i\leq d$ and $0\leq j\leq i-2$. On the one hand, by descending induction hypothesis, for any $i\leq m\leq d$, one has
\begin{equation} {p_{G_i\times \underline{X^{i}}}}_{\;\ast}\left( 
w^i_{m-i}   \cdot z^i_{n-i-j} \cdot \eta_i \right)
=\text{sym}\left((\times_{\substack{ k=0\\ k\neq j  }}^{i-1} h^k)\times h^m\right).\end{equation}
Therefore, the coordinate of (3.3) on top right $h^{i-1}$, i.e., 
\[{p_{\underline{X^{i-1}}\times X}}_{\;\ast}\left(\left({p_{G_i\times \underline{X^{i}}}}_{\;\ast}\left( 
w^i_{m-i}   \cdot z^i_{n-i-j} \cdot \eta_i \right)\right)\cdot [X^{i-1}]\times l_{i-1}\right)\]
is equal to
\begin{equation}\text{sym}\left((\times_{\substack{ k=0\\ k\neq j  }}^{i-2} h^k)\times h^m\right).\end{equation}
On the other hand, by Lemma 2.5, this coordinate is also equal to
\begin{equation}\sum_{k=1}^{\text{min}(m-i+1,i)}
 {p_{G_{i-1}\times \underline{X^{i-1}}}}_{\;\ast}\left( 
w^{i-1}_{m-i+1-k}\cdot \sigma_{i-1}^k   \cdot z^{i-1}_{n-i+1-j} \cdot \eta_{i-1}  \right).\end{equation}
%Let us denote by $V\mathbbm{1}$ the trivial vector bundle on $G_{i-1}$.
Let us denote by $T_{i-1}$ the tautological vector bundle on $G_{i-1}$,
i.e., $T_{i-1}$ is given by the closed subvariety of the trivial bundle
$V\mathbbm{1}=V_q\times G_{i-1}$ consisting of pairs $(u,U)$ such that $u\in U$.
Note that the vector bundle $T_{i-1}$ has rank $i$.
For a vector bundle $E$ over a scheme, we write $c_i(E)$ for the $i$-th Chern class with value in $\text{CH}$.
Since 
$W^{i-1}_{m-i+1-k}=c_{m-i+1-k}(V\mathbbm{1}/T_{i-1})$
(see \cite[Proposition 2.1]{uINV}) and $\sigma_{i-1}^k =$$c_k(T_{i-1})$$\;(\text{mod}\;2)$ (see
\cite[Lemma 2.6]{sym}), by Whitney Sum Formula (see \cite[Proposition 54.7]{EKM}), one has
\begin{multline} \sum_{k=0}^{\text{min}(m-i+1,i)}w^{i-1}_{m-i+1-k}\cdot \sigma_{i-1}^k  
 = \\ \sum_{k=0}^{\text{min}(m-i+1,i)} c_{m-i+1-k}(V\mathbbm{1}/T_{i-1})\cdot c_k(T_{i-1}) = \\
   c_{m-i+1}(V\mathbbm{1}).\end{multline}
Moreover, one has $c_{m-i+1}(V\mathbbm{1})=0$ because $m-i+1>0$. 
Consequently, in view of (3.4), (3.5) and (3.6), one get
\begin{equation}{p_{G_{i-1}\times \underline{X^{i-1}}}}_{\;\ast}\left( 
w^{i-1}_{m-i+1} \cdot z^{i-1}_{n-i+1-j} \cdot \eta_{i-1}  \right)=\text{sym}\left((\times_{\substack{ k=0\\ k\neq j  }}^{i-2} h^k)\times h^m\right).\end{equation}
By identities (3.2) and (3.7), it only remains to prove the case $m=i-1$
to complete the descending induction step.
%\[{p_{G_{i-1}\times \underline{X^{i-1}}}}_{\;\ast}\left( 
 %z^{i-1}_{n-i+1} \cdot \eta_{i-1}  \right)=\text{sym}\left(\times_{j=1}^{i-1}h^{j}\right)\]
%to complete the backward induction step.
On the one hand, by descending induction hypothesis, 
the coordinate of $ {p_{G_i\times \underline{X^{i}}}}_{\;\ast}\left( 
 z^i_{n-i-j} \cdot \eta_i \right)$ on top right $h^{i}$ is
\[\text{sym}\left(\times_{\substack{ k=0\\ k\neq j  }}^{i-1} h^k\right)\] 
(see (3.3)) and,
on the other hand, by Lemma 2.5, it is also equal to 
\[{p_{G_{i-1}\times \underline{X^{i-1}}}}_{\;\ast}\left( 
 z^{i-1}_{n-i+1-j} \cdot \eta_{i-1}  \right).\]
Proposition 3.1 is proven.
\end{proof}

As a consequence of Proposition 3.1, we partially obtain the first part of Theorem 1.2.

\begin{cor}
 \textit{Let $1\leq i \leq d$ and $0\leq j \leq i-1$. If the cycle $z^i_{n-i-j}$ is rational then the cycle} $\rho_{i,j}\;(\text{mod}\;2)$ \textit{is also rational.}
\end{cor}

\begin{proof}[Proof]
In view of the ring structure of $\text{CH}(X_{K}^{i+1})$ (see \cite[\S 68]{EKM}), and
knowing that the cycle $\alpha_{i,j}$ is symmetric, one
deduces from Proposition 3.1 that
\[\alpha_{i,j}\;(\text{mod}\;2)=\Delta_{i,j}+\beta\]
with $\beta$ a sum of nonessential elements (a nonessential element is an external product of powers of the hyperplane class, it is always rational).
Since $\alpha_{i,j}= (\theta_i)_{\ast}(Z^i_{n-i-j})+\rho_{i,j}$ and $\Delta_{i,j}$ is rational (Lemma 2.1), the corollary is proven. 
\end{proof}

%The following statement is a consequence of Proposition 3.1 and its proof.
%It generalizes \cite[Proposition 3.16]{sym}.

\begin{rem}
As a consequence of Proposition 3.1 and its proof, one make the following observation.
Let $1\leq i\leq d-1$, $0\leq j \leq i-1$, $i+1\leq m \leq d$ and $s \in \{0,1,\dots , i\}\backslash \{j\}$.
For any integers $0\leq a_1\leq a_2 \leq \cdots \leq a_i \leq d$, the integer
\[\text{deg}\left(\left(Z^i_{n-i-j}\cdot \prod_{l=1}^i Z^i_{n-i-a_l}\right) \cdot \left(
\sum_{k=0}^{i-s}W^i_{m-s-k}\cdot c_k(T_i)  \right)  \right)\]
is congruent to $1\,(\text{mod}\,2)$ if $\{a_1, \dots , a_i\}=\{m\} \cup (\{0,1,\dots , i\}\backslash \{j,s\})$
and to $0\,(\text{mod}\,2)$ otherwise.
\end{rem}

The following proposition will complete the first part of Theorem 1.2 (see Corollary 3.12).

\begin{prop}
\textit{For any $i\leq j \leq d$, one has}
\[\left((\theta_i)_{\ast}(z^i_{n-i-j})\right)_{\ast}(h^m)=\left\{\begin{array}{ll} 
\text{sym}\left(\times_{k=0}^{i-1}h^k\right) & \textit{if} \:\: m=j  \textit{;}\\
0 & \textit{otherwise.} 
\end{array} \right. \]
\end{prop}

\begin{proof}[Proof]
We already know from Lemma 2.4 that $\left((\theta_i)_{\ast}(z^i_{n-i-j})\right)_{\ast}(h^m)=0$
for $m\leq i-1$.
We prove the cases $i\leq m\leq d$ by descending induction on $i$.
The base of the descending induction (so $i=j=m=d$) is done similarly as  
the base of the descending induction in the proof of Proposition 3.1.

Let $2\leq i \leq d$, $i\leq j \leq d$ and $i\leq m\leq d$. On the one hand, by Lemma 2.4 and the 
descending induction hypothesis, one has 
\begin{equation} \begin{array}{lll} \left((\theta_i)_{\ast}(z^i_{n-i-j})\right)_{\ast}(h^m) & = &
 {p_{G_i\times \underline{X^{i}}}}_{\;\ast}\left( 
w^i_{m-i}   \cdot z^i_{n-i-j} \cdot \eta_i \right) \\
& & \\
& = & \left\{\begin{array}{ll} 
\text{sym}\left(\times_{k=0}^{i-1}h^k\right) & \text{if} \:\: m=j  \textit{;}\\
0 & \text{otherwise.} 
\end{array} \right. \end{array}\end{equation}
Therefore, the coordinate of (3.11) on top right $h^{i-1}$ is
\[
\left\{\begin{array}{ll} 
\text{sym}\left(\times_{k=0}^{i-2}h^k\right) & \text{if} \:\: m=j  \textit{;}\\
0 & \text{otherwise.} 
\end{array} \right.
\]
On the other hand, by Lemmas 2.4, 2.5 and identity (3.6), this coordinate is also equal to
\[\left((\theta_{i-1})_{\ast}(z^{i-1}_{n-i+1-j})\right)_{\ast}(h^m).\]
It remains to consider the cases $i\leq j\leq d$ with $m=i-1$
and $j=i-1$ with $i-1\leq m \leq d$ to complete the descending induction step.
Let $i-1\leq j \leq d$. By Lemma 2.4, one has 
\[\left((\theta_{i-1})_{\ast}(z^{i-1}_{n-i+1-j})\right)_{\ast}(h^{i-1})=
 {p_{G_{i-1}\times \underline{X^{i-1}}}}_{\;\ast}\left( 
 z^{i-1}_{n-i+1-j} \cdot \eta_{i-1} \right).\]
By Lemma 2.5, the latter cycle is the coordinate on top right $h^i$
of $\left((\theta_i)_{\ast}(z^i_{n-i-j})\right)_{\ast}(h^i)$.
If $j\geq i$ then this coordinate is trivial by the descending induction hypothesis.
Otherwise -- if $j=i-1$ -- then one has 
\[ \left((\theta_i)_{\ast}(z^i_{n-2i+1})\right)_{\ast}(h^i)={\rho_{i,i-1}}_{\ast}(h^i)   
+\left( \alpha_{i,i-1} \;(\text{mod}\;2)\right)_{\ast}(h^i).\]
By Proposition 3.1, the latter cycle is equal to
\[\text{sym}\left((\times_{k=0}^{i-2}h^k)\times h^i\right),\]
whose coordinate on top right $h^i$ is 
\[\text{sym}\left(\times_{k=0}^{i-2}h^k\right).\]
Now suppose that $j=i-1$ and let $i\leq m \leq d$.
By Lemmas 2.4, 2.5 and identity (3.6), the cycle 
$\left((\theta_{i-1})_{\ast}(z^{i-1}_{n-2i+2})\right)_{\ast}(h^{m})$ 
is the coordinate of  $\left((\theta_{i})_{\ast}(z^{i}_{n-2i+1})\right)_{\ast}(h^{m})$ on top right $h^{i-1}$.
Since
\[ \left((\theta_i)_{\ast}(z^i_{n-2i+1})\right)_{\ast}(h^m)={\rho_{i,i-1}}_{\ast}(h^m)   
+\left( \alpha_{i,i-1} \;(\text{mod}\;2)\right)_{\ast}(h^m),\]
${\rho_{i,i-1}}_{\ast}(h^m)=0 $ (because $m \neq i-1$) and $\left( \alpha_{i,i-1} \;(\text{mod}\;2)\right)_{\ast}(h^m)=\text{sym}\left((\times_{k=0}^{i-2}h^k)\times h^m\right)$ (Proposition 3.1), this coordinate is trivial.
This completes the descending induction step. The proposistion is proven.
\end{proof}

\begin{cor}
 \textit{Let $1\leq i \leq d$ and $i\leq j \leq d$. If the cycle $z^i_{n-i-j}$ is rational then the cycle} $\rho_{i,j}\;(\text{mod}\;2)$ \textit{is also rational.}
\end{cor}

\begin{proof}[Proof]
In view of the ring structure of $\text{CH}(X_{K}^{i+1})$  and
knowing that the cycle $(\theta_i)_{\ast}(Z^i_{n-i-j})$ is symmetric, one
deduces from Proposition 3.10 that
\[(\theta_i)_{\ast}(z^i_{n-i-j})=\rho_{i,j}\;(\text{mod}\;2) +\beta\]
with $\beta$ a sum of nonessential elements. The corollary is proven. 
\end{proof}

The next proposition gives the second part of Theorem 1.2.

\begin{prop}
 \textit{Let $1\leq i \leq d$ and $0\leq j \leq d$. If the cycle} $\rho_{i,j}$ \textit{is rational then 
the cycle $Z^i_{n-i-j}$ is also rational.}
\end{prop}

\begin{proof}[Proof]
Since ${\pi_{(0,\underline{i})}}_{\ast} \circ \pi_{(\underline{0},i)}^{\ast}(h^i)=[G_i]$ in $\text{CH}^0(G_i)$,
%(see \cite[Proposition 2.1]{uINV}), 
by dimensional reasons, one has
\begin{multline}
({\pi_{(0,\underline{i})}}_{\ast} \circ \pi_{(\underline{0},i)}^{\ast})^{\times i+1}
\left(h^i\times h^{i-1}\times \cdots \times 1
\cdot \text{sym}\left((\times_{k=0}^{i-1}h^k)\times l_j\right)
\right)= \\ [G_i]^{\times i}\times Z^{i}_{n-i-j}.
\end{multline}
The conclusion follows by taking the image of cycle (3.14) under the pull-back of the diagonal morphism $X\rightarrow X^{i+1}$.
\end{proof}

\section{Witt index}

We continue to use notation and material introduced in the previous sections.
In this section, we assume that the $F$-quadric $X$ is anisotropic and we study some
restrictions on the Elementary Discrete Invariant when the first  
Witt index $i_1$ of $X$ is sufficiently large,
using the non-compact form (Theorem 1.2).

% and the  \textit{$1$-primordial cycle} in $\text{Ch}^{n-i_1+1}(X_K^2)$
%(see \cite[Definition 73.16]{EKM} and paragraph
%right after \cite[Theorem 73.26]{EKM}, it is a rational cycle). 

\medskip

By \cite[Lemmas 73.18 and 73.3]{EKM},
there exists a unique minimal rational cycle in 
$\text{Ch}^{n-i_1+1}(X^{2}_K)$ containing $1\times l_{i_1-1}$.
This cycle is symmetric (\cite[Lemma 73.17]{EKM}) and is called the \textit{$1$-primordial cycle}. We denote it by $\pi$.

\begin{prop}
\textit{Let $i \in \{1,\dots, d\}$. Suppose that the quadric $X$ is anisotropic with $i_{1}>i$.
If $m\in EDI(X,i)$ is such that $n-m\notin \{i_1\}\cup \{2i_1,\dots, d+1\}$
then $m+1\in EDI(X,i-1)$.}
\end{prop}

\begin{proof}[Proof]
We set $j=n-i-m$. Since $m\in EDI(X,i)$, the cycle  $\rho_{i,j}\;(\text{mod}\;2)$ is rational by Theorem 1.2. 
We claim that the hypothesis  $i_1> i$ implies that the cycle $\rho_{i-1,j}\;(\text{mod}\;2)$ is also rational,
which, by Theorem 1.2, gives the conclusion.

%Let us denote by $\pi \in \text{Ch}^{n-i_1+1}(X_K^2)$ 
%the $1$-primordial cycle.
%(see \cite[Definition 73.16]{EKM} and paragraph
%right after \cite[Theorem 73.26]{EKM}, it is a rational cycle).
%Even if it means adding a rational cycle to $\pi$, one can assume that 
The rational cycle $\pi$
decomposes as 
\[\pi=1\times l_{i_1-1}+l_{i_1-1}\times 1+\sum_{k=i_1}^{d-i_1+1}a_k\left( h^k \times l_{k+i_1-1}
+ l_{k+i_1-1}\times h^k\right)\]
for some $a_k\in \mathbb{Z}/2\mathbb{Z}$. The fact that one can choose to make the previous sum start from $k=i_1$
is due to \cite[Proposition 73.27]{EKM}.
Since $i_1> i$, 
and $j\notin \{i_1-i\}\cup \{2i_1-i,\dots, d-i+1\}$,
one has
\[\left(\rho_{i,j}\;(\text{mod}\;2)\right)\circ \left((1\times h^{i_1-i})\cdot \pi\right)=1\times \left(\rho_{i-1,j}\;(\text{mod}\;2)\right),\]
where $\circ$ stands for the composition of correspondences.
Therefore, pulling back the latter identity with respect to the diagonal morphism $\delta_i$,
%$X^i\rightarrow X^{i+1}$, $(x_1,x_2,\cdots ,x_i)\mapsto (x_1,x_1,x_2,\cdots ,x_i)$,
one get that the cycle $\rho_{i-1,j}\;(\text{mod}\;2)$ is rational.
\end{proof}

The following statement is obtained by applying recursively Proposition 5.1.

\begin{cor}
\textit{Let $i \in \{1,\dots, d\}$. Suppose that the quadric $X$ is anisotropic with $i_{1}>i$. One has}
\begin{enumerate}[(i)]
\item \textit{if $m\in EDI(X,i)$ then $n-m\geq i_1$;}
\item \textit{if $m\in EDI(X,i)$ and $n-m=i_1+l$ or $d+1+l$ for some $1\leq l <i$ 
then $m+l\in EDI(X,i-l)$.}
\end{enumerate}
\end{cor}

\bibliographystyle{acm} 
\bibliography{references}
\end{document}